\documentclass[11pt]{amsart}
\usepackage[colorlinks=true,citecolor=blue,linkcolor=blue, backref=page]{hyperref}
\usepackage{amssymb,verbatim,amscd,amsmath,graphicx}
\usepackage{enumerate}
\usepackage{graphicx}
\usepackage{caption}
\usepackage{subcaption}
\usepackage{pdfsync}
\usepackage{color,colortbl}
\usepackage{fullpage}
\usepackage{bbm}
\usepackage{comment}
\usepackage{multirow}
\usepackage{enumitem}
\usepackage{comment}
\usepackage[dvipsnames]{xcolor}
\numberwithin{equation}{section}
\newtheorem*{introthm*}{Main Result}
\newtheorem{theorem}{Theorem}[section]
\newtheorem{lemma}[theorem]{Lemma}

\newtheorem{corollary}[theorem]{Corollary}

\theoremstyle{definition}
\newtheorem{definition}[theorem]{Definition}
\newtheorem{conjecture}[theorem]{Conjecture}
\newtheorem{def-prop}[theorem]{Definition-Proposition}

\newtheorem{remark}[theorem]{Remark}

\DeclareMathOperator{\reg}{reg}
\DeclareMathOperator{\depth}{depth}

\DeclareMathOperator{\Ass}{Ass}

\DeclareMathOperator{\HF}{HF}
\DeclareMathOperator{\HP}{HP}

\renewcommand{\AA}{{\mathbb A}}

\newcommand{\PP}{{\mathbb P}}
\newcommand{\ZZ}{{\mathbb Z}}
\newcommand{\NN}{{\mathbb N}}

\newcommand{\XX}{{\mathbb X}}
\newcommand{\YY}{{\mathbb Y}}
\newcommand{\kk}{{\mathbbm k}}

\def\mm{{\frak m}}

\def\pp{{\frak p}}

\def\a{{\bf a}}

\def\p{{\bf p}}

\def\z{{\bf z}}

\def\ahat{\widehat{\alpha}}

\def\1{{\bf 1}}
\def\0{{\bf 0}}

\begin{document}

\title{Lower bounds for Waldschmidt constants and Demailly's Conjecture for general and very general points}

\author{Sankhaneel Bisui}
\address{Arizona State University \\ School of Mathematical and Statistical Sciences \\
	Tempe, AZ 85287-1804, USA}
\email{sankhaneel.bisui@asu.edu, sbisui@tulane.edu} 

\author{Th\'ai Th\`anh Nguy$\tilde{\text{\^E}}$n}
\address{University of Dayton, Department of Mathematics,
	300 College Park, Dayton, Ohio, USA \\
	and University of Education, Hue University, 34 Le Loi, Hue, Vietnam}
\email{tnguyen5@udayton.edu}

\keywords{Demailly's Conjecture, Cremona Transformation, Waldschmidt Constant, Ideals of Points, Symbolic Powers, Containment Problem, Stable Harbourne-Huneke Conjecture}
\subjclass[2010]{14N20, 13F20, 14C20}

\begin{abstract}
We prove Demailly's Conjecture concerning the lower bound for the Waldschmidt constant in terms of the initial degree of the second symbolic powers for any set of  very general points in $\PP^N$. We also discuss the Harbourne-Huneke Containment and the aforementioned Demailly's Conjecture for general points and show the results for sufficiently many general points and general points in projective spaces with low dimensions. 
\end{abstract}
\maketitle

\section{Introduction}
\label{sec.intro}
Given a set of $s$ reduced points  $\XX=\{P_1, \dots, P_s\}$ in $\PP^N_\kk$, over an algebraically closed field with coordinate ring $R=\kk[\PP^N_\kk]$, one of the fundamental questions asks about the least degree of a homogeneous polynomial in $R$ vanishing at $\XX$ with  multiplicity at least $m$. If  $I_\XX$ denotes the ideal defining $\XX$, then by Zariski-Nagata theorem, the $m$-th symbolic power of $I_\XX$, defined as $I^{(m)}=\cap_{i=1}^s\p_i^m$, where $\p_i$ is the ideal defining the point $P_i$, consists of  all homogeneous polynomials vanishing at $\XX$ with multiplicity at least $m$. Let $\alpha(I)$ be the least degree of elements in $I$ (or initial degree of $I$). The Waldschmidt constant of $I$ is defined as $\ahat(I)= {\displaystyle \lim_{m\to 0}}\frac{\alpha(I^{(m)})}{m}$. 
Demailly proposed the following lower bound for the Waldschmidt constant of the defining ideal of points $I$ which simultaneously gives lower bounds for all $\alpha(I^{(m)})$ (see Definition \ref{def.waldconst}).
\begin{conjecture}\cite{Demailly1982}\label{conjecture.Demailly}
If  $I_\XX$  is the defining ideal of $\XX=\{P_1,  \dots, P_s\} \subset \mathbb{P}^N_{\mathbb{C}} $, then
	$${\ahat(I_\XX)} \geqslant \dfrac{\alpha(I_\XX^{(m)})+N-1}{m+N-1}, \text{ for all } m \geqslant 1.$$
\end{conjecture} 
The above inequality is sharp: it becomes equality when, for instance, $I$ is a complete intersection. It is also asymptotically an equality (when $m\rightarrow \infty$).  For $m=1$, Conjecture \ref{conjecture.Demailly} is known as the Chudnovsky's conjecture \cite{Chudnovsky1981} and is well studied, see, \cite{Chudnovsky1981, HaHu, dumnicki2012symbolic, Dumnicki2015,  BoH, GHM2013, FMX2018, DTG2017, FMX2018, bghn2021chudnovskys,SankhoThaiChudnovsky}.  
\par 
Esnault and Viehweg \cite{EsnaultViehweg} verified Conjecture \ref{conjecture.Demailly} for $N = 2$.  Bocci, Cooper, and Harbourne \cite{BCH2014} provided a more elementary proof using ideal containment for $s=n^2$ general points, and for at least $(m+1)^2-1$ general points for any fixed $m$ in $\PP^2$. In dimesion $N\geqslant 3$, the conjecture is wide open.
Malara, Szemberg, and Szpond \cite{MSS2018} proved that for any fixed integer $m$, any set of at least $(m+1)^N$ very general points satisfies Demailly's conjectural inequality. This was later improved to at least $m^N$ very general points by the work of Dumnicki, Szemberg, and Szpond \cite{localeffectivity}. Trok and Nagel \cite{NagelTrokInterpolation} proved that Conjecture \ref{conjecture.Demailly} holds for at most $N+2$ and specific cases of $N+3$ general points. Chang and Jow \cite{chang2020demailly} verified the conjecture for $s=k^N$ general points, where $k\in \NN$. The authors along with Grifo and Hà \cite{bghn2022demailly} proved that for any fixed $m\in \NN$, any set of at least $(2m+2)^N$ general points verifies Demailly's conjecture. It was also shown to be true for special configurations of points including the star configurations of points \cite{MSS2018, bghn2022demailly}, (a more elementary proof for) the Fermat configurations and other configurations of points in $\PP^2$ \cite{Ng23b}, and even beyond the points setting including determinantal ideals \cite{bghn2022demailly}, star configurations of hypersurfaces \cite{bghn2022demailly} (also known as the uniform $a-$fold product ideal, see, Lin and Shen \cite{LinShen22}), and special configurations of lines \cite{Ng23a}. \par
Besides classical applications in complex analysis, see for instance \cite{Waldschmidt,Moreau1980}, Demailly's bound was shown to be beneficial for the study of the ideal containment problem \cite{YuXieStefan}, as well as computing the values $\alpha(I_\XX^{(m)})$ for special configurations \cite{Ng23a,Ng23b}. Demailly's bound also gives rise to upper bounds for the regularity values of a quotient by ideals generated by powers of linear forms \cite{DNS23}, which is useful to study dimensions of certain spline spaces \cite{DiPVill21}.

 We briefly recall the descriptions of  very general and general points. We follow the exposition from the work of Fouli, Mantero, and Xie \cite[Page 4]{FMX2018}.  A set of points $\XX=\{P_1, \dots, P_s\}$ is an element of the Chow variety of algebraic $0$-cycles of degree $s$ in $\PP^N$ (or an element in the Hilbert scheme of $s$ points in $\PP^N$). A property $\mathcal{P}$ holds for $s$ \emph{very general} set of points in $\PP^N_\kk$, if there exist countable infinite  open dense subsets $U_t$, $t \in \NN$, of  the Chow variety of algebraic $0$-cycles of degree $s$  such that the property $\mathcal{P}$ holds for all $\XX \in \bigcap_{t=1}^\infty U_t$. If there exists one open dense subset $U$ such that the property $\mathcal{P}$ holds for all $\XX \in U$, then the property $\mathcal{P}$ holds for  $s$ \emph{general} sets of points. \par 
 
\vspace{0.5em}
In this manuscript, we focus on a specific case of Conjecture \ref{conjecture.Demailly} when $m=2$ and $N\geqslant 3$.  
\begin{conjecture}\label{conjecture.Demailly2}
If  $I_\XX$  is the defining ideal of $\XX=\{P_1,  \dots, P_s\} \subset \mathbb{P}^N_{\mathbb{C}} $, then
\[{\ahat(I)}\geqslant \dfrac{\alpha(I^{(2)}) + N-1}{N+1}. \label{eq.Demailly}
\]
\end{conjecture}

As mentioned before, Conjecture \ref{conjecture.Demailly2} was proved for at least $2^N$ very general points \cite{localeffectivity}, for at most $N+2$ and some cases of $N+3$ general points \cite{NagelTrokInterpolation}, and for at least $6^3$ general points in $\PP^3$, at least $5^4$ general points in $\PP^4$, and at least $4^N$ general points in $\PP^N$ where $N\geqslant 5$ \cite{bghn2022demailly}. 

Our main result is to complete the proof of Conjecture \ref{conjecture.Demailly2} for any set of very general points. In addition, we verify the conjecture for at least $2^N$ general points, and general points in $\PP^3,\PP^4$, and  $\PP^5$ except for four cases in $\PP^5$. This is not a routine checking as it involves finding better bounds for the Waldschmidt constants in given cases. In general, proving the results for general points is much harder than that for very general points as we have to move from countably infinite open conditions to one open condition. 

\begin{introthm*}
Conjecture \ref{conjecture.Demailly2} holds for the points in the following list (*).
\begin{itemize}
    \item (Theorem \ref{thm.Demaillyverygeneral}) $s$ very general points where $N+3\leqslant s \leqslant 2^N$. 
\item  (Corollary \ref{cor.generalatleast2^N}) at least $2^N$ general points. 
\item  (Corollary \ref{cor.Demailly2.general.P345}) $s$ general points in $\PP^N$, for the following cases: 
            \begin{enumerate}[label=(\roman*)] 
                          \item $6 \leqslant s \leqslant 6^3$ ,  in $\PP^3$
                          \item  $8 \leqslant s \leqslant 5^4$, in $\PP^4 $ 
                          \item  $s=8,9$ and $14 \leqslant s \leqslant   4^5$, in $\PP^5$.
                          \end{enumerate}
\end{itemize}

Coupled with results in \cite{bghn2022demailly,localeffectivity,NagelTrokInterpolation}, Demailly's Conjecture when $m=2$ now holds for any set of very general points, any set of at least $2^N$ general points, and any set of general points in $\PP^3,\PP^4,$ and $\PP^5$, except for $10\leqslant s \leqslant 13$ general points in $\PP^5$.  
\end{introthm*}
The result for $10\leqslant s \leqslant 13$ general points in $\PP^5$ is  still not known. We believe a more clever treatment for the Waldschmidt constant of $10$ points would lead to the affirmative answer to these four cases. We now briefly discuss our strategies. This involves deriving and using regularity and the initial degree of the second symbolic power of the defining ideal of generic points from its Hilbert function. 
Let $\p_i$ be the defining ideal of the point $P_i \in\PP^N$, then the fat point scheme $\YY=m_1P_1+\dots + m_sP_s$ is defined by the ideal $I_\YY=\p^{m_1}\cap \dots \cap \p^{m_s}$.
The Hilbert function of $\YY$ is defined as follows \[\HF_{R/I_\YY}(t)=\dim_\kk\left( R/I_\YY\right)_t, \text{ for all } t \in \ZZ_+.\] 
The Alexander-Hirschowitz Theorem provides the Hilbert function of  $\YY$, when $\YY=2P_1+\dots +2P_s$.
From the definition, we obtain $\dim_\kk[I_\YY]_t={N+t\choose N}-\HF_{R/I_\YY}(t)$. Thus the least $t>0$ for which $\HF_{R/I_\YY}(t)<{N+t\choose N}$ will give the initial degree $\alpha(I_\YY)$.  
\par For the \emph{very general} case, we control $\alpha(I^{(2)})$ using the Alexander-Hirschowitz Theorem and derive in Lemma \ref{lem.Demaillyverygeneral} the following:  if $I$ is the ideal defining $s$ generic points then to show Conjecture \ref{conjecture.Demailly2} it is enough to prove \[\ahat(I)\geqslant \dfrac{N+\ell}{N+1}, \text{ where } {N+\ell \choose N} \leqslant (N+1)s < {{N+\ell+1} \choose N}.\tag{1}\label{inq1}\]
Next, to handle $\ahat(I)$, in Theorem \ref{thm.Demaillyverygeneral}, we split the number of points into two adequate parts to use induction on $N$, reduction methods  and results from \cite{localeffectivity,SankhoThaiChudnovsky} to obtain the proper lower bound for $\ahat(I)$ to prove inequality \ref{inq1} and  thus  prove Conjecture \ref{conjecture.Demailly2} for very general points. \par
\vspace{0.5em}
For \emph{general} points, we need an additional tool. To prove Conjecture \ref{conjecture.Demailly2} for general sets of points in the List (*), we prove the containment $I^{(r(N+1)-N+1)} \subseteq \mm^{r(N-1)}\left ( I^{(2)} \right )^r$, for some $r\in \NN$,  for ideals defining generic points, as in  \cite[Theorem 2.9]{bghn2022demailly}. Following the strategy in \cite{BCH2014, bghn2022demailly}), we show that a specific inequality between the Waldschmidt constant and the regularity of the second symbolic power would imply the aforementioned containment. \par
\vspace{0.5em}

\noindent\textbf{Theorem \ref{thm.neededineq}} Let $I \subset R=\kk[\PP^N]$ be the defining ideal of $s$ generic points. If 
\[
\widehat{\alpha}(I) \geqslant \frac{N-1+\reg(I^{(2)})}{N+1},
\]
then there exists $r$ such that 
$I^{(r(N+1)-N+1)} \subseteq \mm^{r(N-1)}\left ( I^{(2)} \right )^r$. \par
\vspace{0.5em}

Thus, we can focus on proving the hypothesis inequality of the above theorem for each case of general points. Firstly, upper bounds for the regularity of the second symbolic power of the ideal defining generic points can be derived from the Alexander-Hirschowitz Theorem.  
Recall that the regularity of $\YY$ is the least integer $t+1$ such that $ \HF_{R/I_\YY}(t)=\sum_{i=1}^{s}{N+m_i-1 \choose N}$. 
In Lemma \ref{lem.Demaillygeneral},  we derive the upper bounds and prove that:   if $I$ is the defining ideal of $s$ generic points then to prove the hypothesis of  Theorem \ref{thm.neededineq} it is sufficient to prove
\[\ahat(I)\geqslant \dfrac{N+\ell+1}{N+1}, \text{ where }  {N+\ell \choose N} < (N+1)s \leqslant {{N+\ell+1} \choose N}. \tag{2}\label{inq2}\]

Secondly, we need the proper lower bounds for the Waldschmidt constant of the defining ideal of generic points for proving inequality  \ref{inq2}. Using reduction methods from \cite{SankhoThaiChudnovsky}, we establish reasonable lower bounds for the Waldschmidt constant in Theorem \ref{thm.Nge62^Nto4^N}, which leads to Corollary \ref{cor.generalatleast2^N}. Again, using reduction methods and splitting the points, as in \cite{localeffectivity, SankhoThaiChudnovsky},  we establish reasonable lower bounds for the Waldschmidt constant  in Theorems \ref{thm.needediqP3}, \ref{thm.needediqP4}, and \ref{thm.needediqP5}, which lead to Corollary \ref{cor.Demailly2.general.P345}.  Thus combining we prove Conjecture \ref{conjecture.Demailly2} for   $s$ many general points in the List (*). 
  
 We believe that our strategies can be used to prove Demailly's Conjecture \ref{conjecture.Demailly} for higher values of $m$. Following our method, one can derive similar inequalities as (\ref{inq1}) and (\ref{inq2}).  After that computation of  adequate lower bounds for the Waldschmidt constant using methods in \cite{localeffectivity, SankhoThaiChudnovsky} or other methods would lead to the desired conjectural bound \ref{conjecture.Demailly}. 

This paper is outlined as follows. In Section \ref{sec.prelim}, we recall all results needed for later sections.  
In Section \ref{sec.DemVeryGen}, we prove Conjecture \ref{conjecture.Demailly2} for any set of very general points. In Section \ref{sec.>2^NGen} we prove Conjecture \ref{conjecture.Demailly2} for at least $2^N$ many general points in $\PP^N$ with $N \geqslant 6$. In Section \ref{sec.DemGen} we prove that any set of general points in $\PP^N, N=3,4,5$ satisfies Conjecture \ref{conjecture.Demailly2}, except for the case ($N=5, 10 \leqslant s\leqslant 13$).
\par
\vspace{1em}
\noindent
{\bf Acknowledgments.}
The authors thank Huy T\`ai H\`a for 
his comments and  suggestions.
The second author (TTN) is partially supported by the NAFOSTED (Vietnam) under the grant number 101.04-2023.07. Part of this work was completed when the second author was visiting the Vietnam Institute of Advanced Study in Mathematics (VIASM). He thanks VIASM for the hospitality and support. The authors also thank the anonymous referees for their detailed comments, which helped improve the presentation of the manuscript.  

\section{Preliminaries}
\label{sec.prelim}
In this section, we collect important definitions and notations used in this paper.
For unexplained terminology, we refer the interested reader to the text \cite{CHHVT2020}. We will only work over algebraically closed field of charateristic zero. 

 \begin{definition}
 \label{def.symbpower}
 	Let $R$ be a commutative ring and let $I \subseteq R$ be an ideal. For $m \in \NN$, we define the \emph{$m$-th symbolic power} of $I$ as
 	$$I^{(m)} = \bigcap_{\pp \in \Ass(I)} \left(I^mR_\pp \cap R\right).$$
 \end{definition}
 
Although symbolic powers are notoriously hard to compute in general, symbolic powers of ideals of points have a nice description. It is well-known that if $\XX$ is the set $ \{P_1, \dots, P_s\} \subseteq \PP^N_\kk$ of $s$ many distinct points and $\p_i \subseteq \kk[\PP^N_\kk]$ is the defining ideal of $P_i$, then $I = \p_1 \cap \dots \cap \p_s$ is the ideal defining $\XX$, and the $m$-th symbolic power of $I$ is given by 
	$$I^{(m)} = \p_1^m \cap \dots \cap \p_s^m.$$

To study lower bounds for the initial degree of symbolic powers of ideals, the \emph{Waldschmidt constant} turned out to be very useful by the following definition-lemma.

\begin{definition}
\label{def.waldconst}
 	Let $I \subseteq \kk[\PP^N_\kk]$ be a homogeneous ideal and $\alpha(J)$ denote the initial degree of a homogeneous ideal $J$. The \emph{Waldschmidt constant} of $I$ is defined as 
 	$$\ahat(I) := \lim_{m \rightarrow \infty} \dfrac{\alpha(I^{(m)})}{m} = \inf_{m \in \NN} \dfrac{\alpha(I^{(m)})}{m}.$$ 
 	See, for example, \cite[Lemma 2.3.1]{BoH}. 
 \end{definition}
We recall the statement of  Demailly's Conjecture in terms of the Waldschmidt constant.
 \begin{conjecture}[Demailly]
 	If $I \subseteq \kk[\PP^N_\kk]$ is the defining ideal of a set of (reduced) points in $\PP^N_\kk$, then, for each $m\in \NN$,
 	$$\ahat(I) \geqslant \dfrac{\alpha(I^{(m)})+N-1}{m+N-1}.$$
 \end{conjecture}
 
 A direct approach to tackle this conjecture is to treat the two sides of the inequality separately. We will see that the right-hand side, when $m=2$ and the points are general, can be computed by the Alexander-Hirschowitz Theorem, while the left-hand side can be bounded by a combination of reduction techniques. We also define the following notation for later use.
 
 \begin{definition}
 	For $i=1,\ldots s$, let $\p_i $ be the defining ideal of a point $P_i \in \XX=\{P_1, \dots, P_s \} \subset \PP^N_\kk$ and $\overline{m}=(m_1, \dots, m_s)$ be a tuple of positive integers. The fat point scheme denoted by $m_1P_1+m_2P_2+ \dots + m_sP_s$ is the scheme with the defining ideal 
 	$$I (\overline{m})= I(m_1, \dots, m_s) = \p_1^{m_1} \cap \p_2^{m_2} \cap \dots \cap \p_s^{m_s}. $$
 	When $m_j \leqslant 0$, we take $\p_j^{m_j}=\kk[\PP^N_\kk]$. 
 	We will also use the following notation: $$I\left(m^{\times s}\right)=I(\underbrace{m,m, \dots, m}_{ s \text{ times }}).$$
  If we want to emphasize the dimension of the projective space where the points live in, we also use the notation $I (\PP^N,\overline{m})$ or $I (\PP^N,m^{\times s})$.
 \end{definition}

  It is known that the set of all collections of $s$ (not necessarily distinct) points in $\PP^N_\kk$ is parameterized by the Chow variety $G(1,s,N+1)$ of $0$-cycles of degree $s$ in $\PP^N_\kk$ (see \cite{GKZ1994}) (or the Hilbert scheme of $s$ points in $\PP^N$). A property $\mathcal{P}$ is said to hold for a \emph{general} set or \emph{very general} set of $s$ points in $\PP^N_\kk$ if it holds for all configuration of $s$ points in an open dense subset $U \subseteq G(1,s,N+1)$ or in a countable intersection of such open dense sets, respectively. 
  
The paper \cite{FMX2018} offers an alternative way to describe general and very general points that provides a practically useful way to approach Demailly's conjecture and containment problem for general and very general sets of points.
Let $(z_{ij})_{1 \leqslant i \leqslant s, 0 \leqslant j \leqslant N}$ be $s(N+1)$ new indeterminates. We use $\z$ and $\a$ to denote the collections $(z_{ij})_{1 \leqslant i \leqslant s, 0 \leqslant j \leqslant N}$ and $(a_{ij})_{1 \leqslant i \leqslant s, 0 \leqslant j \leqslant N}$, respectively. Let $\kk(\z)$ be the field extension adjoining all $z_{ij}$, 
 $$P_i(\z) := [z_{i0}: \dots : z_{iN}] \in \PP^N_{\kk(\z)}, \quad \text{ and } \quad \XX(\z) := \{P_1(\z), \dots, P_s(\z)\}.$$
 We call the set $\XX(\z)$ the set of $s$ \emph{generic} points in $\PP^N_{\kk(\z)}$. For any $\a \in \AA^{s(N+1)}_\kk$, let $P_i(\a)$ and $\XX(\a)$ be obtained from $P_i(\z)$ and $\XX(\z)$, respectively, by setting $z_{ij} = a_{ij}$ for all $i,j$. There exists an open dense subset $W_0 \subseteq \AA^{s(N+1)}_\kk$ such that $\XX(\a)$ is a set of distinct points in $\PP^N_\kk$ for all $\a \in W_0$, and all subsets of $s$ points in $\PP^N_\kk$ arise in this way. It is shown that if a property $\mathcal{P}$ holds for $\XX(\a)$ whenever $\a$ in an open dense set $U \subseteq \AA^{s(N+1)}_\kk$ (or a countable union of open dense subsets), the property $\mathcal{P}$ holds for a general (or very general, respectively) set of $s$ points in $\PP^N_\kk$, see \cite[Lemma 2.3]{FMX2018}. Moreover, this construction behaves well and compatible with the \emph{specialization} process given in the work of Krull \cite{Krull1948}, see also \cite[Definition 2.8, Remarks 2.9, 2.10]{bghn2021chudnovskys}.
 From now on, we will use the terms general and very general as in this construction. We also collect the following well-known facts about the Waldschmidt constant of a set of generic points for later use.

 \begin{lemma}\label{lemma: known inequalities of Waldschmidt constant}
 	If	$\ahat(s) := \ahat(I (1^{\times s}))$  is the Waldschmidt constant of the defining ideal of $s$ generic points in $\PP^N$, then the followings are true  
 	\begin{enumerate}[label=(\roman*)]
 		\item $\ahat(s) \geqslant \ahat(k)$ whenever $s \geqslant k$.  
 		\item  $\alpha(I(m^{\times s})) \geqslant m\ahat(s) $.
 		\item $ \ahat \big(I \big( 1 ^{\times {k^N}}\big) \big) =k$.\cite{EVAIN2005516, DTG2017}
            \item $\ahat(s) \geqslant \ahat(J)$ for any $J$ defining ideal of $s$ points.
 	\end{enumerate}
 \end{lemma}

 Our proof of Demailly's Conjecture when $m=2$ relies on the lower bound for the Waldschmidt constant of a set of generic points and the upper bound of the Castelnouvo-Mumford regularity and the initial degree of the second symbolic power of the defining ideal of the set of generic points. To study the former, we utilize the \emph{Cremona reduction} process, see for instance, \cite{dumnicki2012symbolic,Dumnicki2015,SankhoThaiChudnovsky}, and gluing results from the Waldschmidt decomposition technique given in \cite{localeffectivity}. For the latter, the celebrated Alexander-Hirschowitz Theorem provides bounds for the regularity value and the initial degree of the second symbolic power. We first recall some useful results that can be proved by using Cremona transformations.


\begin{lemma}{\cite[Lemma 3.1]{SankhoThaiChudnovsky},\cite[Proposition 8]{dumnicki2012symbolic}} 
\label{lemma: reduction of multiplicity of points}
	Let $I(m_1, \dots, m_s)$ be the ideal of $s$ generic points or general points with multiplicities $m_1, \dots, m_s$. If $$I(m_1, \dots, m_s)_d \neq 0, \text{ then    } I (m_1+k, \dots,m_{N+1}+k,m_{N+2}, 
	\dots,  m_s ) _{d+k} \neq 0,$$ where $k = (N-1)d-\sum_{j=1}^{N+1}m_j $. 
\end{lemma}

 \begin{lemma}{\cite[Proposition 10]{Dumnicki2015}} 
 \label{lemma: adding multiplicities Dumnicki}
 	Let  $m_1, \dots, m_r, m_1', \dots, m'_s, t,k$ be integers. If $I(m_1, \dots ,m_r)_k =0$ and 	$I(m_1', \dots, m'_s, k+1 )_t =0,$
 	then $I(m_1, \dots, m_r, m_1', \dots, m'_s )_t=0 .$
 \end{lemma}

 We also have the following useful result whose proof uses the Waldschmidt decompositions.

 \begin{lemma}[Theorem 4.1 \cite{localeffectivity}]\label{lemma: Waldschmidtdecomp}
Denote $\ahat(\PP^N, r)$ the Waldschmidt constant of the ideal of $r$ very general points in $\PP^N$. Let $N\geqslant 2$ and $k\geqslant 1$. Assume that for some integers $r_1,\ldots ,r_{k+1}$ and rational numbers $a_1,\ldots ,a_{k+1}$ we have 
$$\ahat(\PP^{N-1}, r_j) \geqslant a_j, \text{ for all } j=1,\ldots ,k+1,$$
$$k\leqslant a_j \leqslant k+1, \text{ for } j=1,\ldots ,k, \ \  a_1>k, \ \  a_{k+1}\leqslant k+1.$$
Then, $$\ahat(\PP^{N}, r_1+\ldots +r_{k+1}) \geqslant \left(1 - \sum_{j=1}^k \dfrac{1}{a_j} \right)a_{k+1} +k.$$
\end{lemma}

Lastly, we recall the celebrated Alexander-Hirschowitz Theorem, which was proved in a series of work \cite{A, AH1, AH2, AH3}, and explain how it gives the regularity value and the initial degree of the second symbolic powers. 

\begin{theorem}
    \label{thm.Alex-Hir}
    Let $I \subset R=\kk[\PP^N]$ be the defining ideal of $s$ general points  in $\PP^N$, then 
    \[
    \HF_{R/I^{(2)}}(d) = \min \left \{ {d+N \choose N}, s(N+1) \right\},
    \]
    where $\HF_{A}$ denotes the Hilbert function of a graded algebra $A$, except the following cases
    \begin{enumerate}[label=(\roman*)]
        \item $d=2, 2\leqslant s \leqslant N$,
        \item $d=3, N=4, s=7$,
        \item $d=4, 2\leqslant N \leqslant4, s = {N+2 \choose 2} -1$.
    \end{enumerate}
\end{theorem}

In our context, the Castelnouvo-Mumford regularity of the defining ideal of fat point schemes can be interpreted as follows. First, for an ideal $J$, we define the {\it postulation number} of $R/J$ to be
\[{\rm post}(R/J) = 
\max\{d \in \mathbb{Z} ~|~ HF_{R/J}(d) \neq HP_{R/J}(d)\}.\]
In other words, ${\rm post}(R/J)$ is  the last value where the Hilbert function and Hilbert function disagree. By the following Serre's formula,

$$HF_{R/J}(i) - HP_{R/J}(i) = \sum_{j= \depth(R/J)}^{\dim(R/J)} (-1)^j \dim_k (H_\mathfrak{m}^j(R/J))_i, \forall i \in \mathbb{Z},$$

when $R/J$ is Cohen-Macaulay, as $\depth(R/J) = \dim(R/J)$, we have the following definition.

\begin{definition}\label{def.regularity}
Let $J \subset R=\kk[\PP^N]$ be  a Cohen-Macaulay homogeneous ideal. The regularity of $R/J$ is  defined by 
$\reg(R/J)={\rm post}(R/J) + \dim(R/J).$
\end{definition}

\begin{lemma}
    \label{lem.regbound}
   Let $I \subset R=\kk[\PP^N]$ be the defining ideal of $s$ general points, we have 
\[
    \reg(R/I^{(2)}) = {\rm post}(R/I^{(2)}) + 1 = \min \left\{d \in \mathbb{Z} ~|~ \HF_{R/I^{(2)}}(d) = \HP_{R/I^{(2)}}(d) = s(N+1)\right\}.
\]
In particular, for each $d$ such that $\HF_{R/I^{(2)}}(d) = s(n+1)$, we have $\reg(I^{(2)}) \leqslant d+1$.
\end{lemma}
\begin{proof}
    The results follow since $\dim(R/I^{(2)}) =1$ and $\reg(I^{(2)}) = \reg(R/I^{(2)}) +1$.
\end{proof}

Similarly, for the initial degree $\alpha(I^{(2)})$, it is clear that 
$$\alpha(I^{(2)}) = \min \left\{d \in \NN ~|~ \HF_{R/I^{(2)}}(d) \ne {d+N \choose N}\right\} \leqslant \min \left\{d \in \NN ~|~ s(N+1) < {d+N \choose N} \right\}.$$
Note that equality holds except for some cases arising from the exceptional cases in Theorem \ref{thm.Alex-Hir}.



\section{Demailly's Conjecture when $m=2$ for Very General Points}
\label{sec.DemVeryGen}
In this section, we show that Demailly's Conjecture when $m=2$ holds for any set of $s$ very general points when $N \geqslant 3$. Demailly's Conjecture was proved completely in $\PP^2$, and when $m=2, N\geqslant 3$, it was proved for $s\leqslant N+2$, see \cite{NagelTrokInterpolation}, and $s\geqslant 2^N$ number of very general points, see \cite{localeffectivity}. We will show the conjecture for all remaining cases of very general points, that is, when $N+3 \leqslant s \leqslant2^N$. We first begin with the following lemma giving a way to control the right-hand side of the conjectural inequality.

\begin{lemma}
    \label{lem.Demaillyverygeneral}
    For a fix $\ell \in \NN$, let $I$ be the defining ideal of a set of $s$ generic points or very general points such that $\displaystyle {N+\ell \choose N} \leqslant (N+1)s < {{N+\ell+1} \choose N}$. If $\displaystyle \ahat(I)\geqslant \dfrac{N+\ell}{N+1},$
    then Demailly's Conjecture when $m=2$ holds for $I$, that is
    $$\ahat(I)\geqslant \dfrac{\alpha(I^{(2)})+N-1}{N+1}.$$
    
\end{lemma}

\begin{proof}
    If ${N+\ell \choose N} \leqslant (N+1)s < {{N+\ell+1} \choose N}$ then by Theorem \ref{thm.Alex-Hir},  $\alpha \left( I^{(2)} \right) \leqslant \ell+1$. 
     Thus the inequality $\ahat(I)\geqslant \frac{N+\ell}{N+1}$ implies Demailly's inequality for each $\ell$.
\end{proof}

\begin{theorem}
    \label{thm.Demaillyverygeneral}
    Let $N\geqslant 3$ and $I \subset R=\kk[\PP^N]$ be the defining ideal of $N+3 \leqslant s \leqslant2^N$ generic points or very general points. Then, the Demailly's Conjecture holds for $m=2$, that is,
     \[
     \widehat{\alpha}(I) \geqslant \frac{N-1+\alpha(I^{(2)})}{N+1}.
     \]
\end{theorem}

\begin{proof}
    The result for very general points follows from specializing the result for generic points, see \cite{FMX2018}. Hence, we will assume that $I$ is the defining ideal of $s$ generic points. Moreover, we will prove the (stronger) results for $N=3,4$ in Section \ref{sec.DemGen}, hence, we will assume here that $N\geqslant 5$. By Lemma \ref{lem.Demaillyverygeneral}, it suffices to show that 
    $$\ahat(I)\geqslant \dfrac{N+\ell}{N+1},$$
    where $\displaystyle {N+\ell \choose N} \leqslant (N+1)s < {{N+\ell+1} \choose N}$ and $\ell \geqslant 2$ as $\displaystyle {N+2 \choose N}<(N+3)(N+1)\leqslant s(N+1)$. \par
    \vspace{0.5em} 
    From \cite[Proposition B.1.1]{brianlinear}, we have $ \ahat(I(1^{\times N+3}))\geqslant \frac{N+2}{N}$. Thus when, $N+3 \leqslant s$, and $ {N+2 \choose N} \leqslant s(N+1) < {N+3\choose N}$ then by Lemma \ref{lemma: known inequalities of Waldschmidt constant}, we have 
    \[
    \ahat(I(1^{\times s}))  \geqslant  \ahat(I(1^{\times N+3})) \geqslant \dfrac{N+2}{N} > \dfrac{N+2}{N+1}.
    \]

    Similarly, when $N+3\leqslant s$, and  $ \displaystyle {N+3 \choose N}\leqslant (N+1)s < {N+4 \choose N}$, we have
    \[
    \ahat(I(1^{\times s}))  \geqslant  \ahat(I(1^{\times N+3})) \geqslant \dfrac{N+2}{N} > \dfrac{N+3}{N+1},
    \]
    as desired. On the other hand, we have $\displaystyle 2^N(N+1) \leqslant {2N \choose N}$ for all $N\geqslant 5$. Therefore, it suffices now to show that 
    $\ahat(I)\geqslant \dfrac{N+\ell}{N+1},$
    where $\displaystyle {N+\ell \choose N} \leqslant (N+1)s < {{N+\ell+1} \choose N}$ and $4 \leqslant \ell \leqslant N-1$.
    We will use induction on $N$. The base cases $N=3$ and $4$ will be proved in Section \ref{sec.DemGen} as mentioned earlier. We consider the following two cases. \par
    \vspace{0.5em}

    \noindent\textsf{Case 1}: When $\ell \leqslant\dfrac{3+\sqrt{4N+5}}{2}$. We first claim that for such $\ell$, we have 
    \[
    {N+\ell \choose N} \geqslant (N+1){N+\ell-2 \choose N}.
    \]
    In fact, the above inequality is equivalent to 
    \begin{equation*}
 (N+\ell)(N+\ell-1)\geqslant (N+1)\ell(\ell-1) \iff   \ell^2-3\ell-(N-1)\leqslant0 \iff \dfrac{3-\sqrt{4N+5}}{2} \leqslant\ell \leqslant\dfrac{3+\sqrt{4N+5}}{2} 
    \end{equation*}
  Thus, the claim follows since we assume $\ell \geqslant 2 > \frac{3-\sqrt{4N+5}}{2}$. 

    Now, for $\ell \leqslant\dfrac{3+\sqrt{4N+5}}{2}$, if $\displaystyle {N+\ell \choose N} \leqslant (N+1)s$, we have that $s\geqslant \displaystyle {N+\ell-2 \choose N}$. Since $\ell \geqslant 4$, we have $\ell -2 \geqslant 2$, thus, by the same proof as that of \cite[Lemma 4.5, Lemma 4.9, Theorem 4.10]{SankhoThaiChudnovsky}, we have that 

    \[
    \ahat(I)\geqslant \dfrac{N+\ell -1}{N} > \dfrac{N+\ell}{N+1}, 
    \]
    as desired.

    \noindent\textsf{Case 2}:  When $\ell > \dfrac{3+\sqrt{4N+5}}{2}$. We first claim that for such $\ell$, we have
    \[
    \dfrac{{N+\ell \choose N}}{N+1} \geqslant \dfrac{{N-1+\ell \choose N-1}}{N} + \dfrac{{N-1+\ell-1 \choose N-1}}{N} +1
    \iff \dfrac{(N+\ell -2)!}{\ell!(N+1)!}(\ell^2-3\ell-N+1)\geqslant 1.  \]
    Note that $\ell^2-3\ell-(N-1)> 0$ if and only if $ \ell > \dfrac{3+\sqrt{4N+5}}{2}$ since $\dfrac{3-\sqrt{4N+5}}{2} <0$ when $N\geqslant 5$. Suppose that $\ell=\ell_1=\left\lfloor\ell_0 \right\rfloor +1$, where $\ell_0= \dfrac{3+\sqrt{4N+5}}{2}$.
    Then \begin{align*}
 &\dfrac{(N+\ell_1 -2)!}{\ell_1!(N+1)!}(\ell_1^2-3\ell_1-N+1)=\dfrac{(N+\ell_1-2)\cdots (N+2)(\ell_1^2-3\ell_1-N+1)}{\ell_1(\ell_1-1)(\ell_1-2)\cdots 3\cdot2\cdot1}\\
 &=\dfrac{(N+\ell_1-2)}{\ell_1}\cdot \dfrac{(N+\ell_1-3)}{\ell_1-1}\cdots \dfrac{(N+4)}{ 6}\cdot\dfrac{(N+3)}{4\cdot 3}\cdot\dfrac{(N+2)}{5 \cdot 2}\cdot(\ell_1^2-3\ell_1-N+1).
    \end{align*}

Since $\ell_1^2-3\ell_1-N+1 >0$ and it is an integer, we have $\ell_1^2-3\ell_1-N+1 \geqslant 1$. If $N\geqslant 9$ then each term in the product is at least $1$, and thus $\dfrac{(N+\ell_1 -2)!}{\ell_1!(N+1)!}(\ell_1^2-3\ell_1-N+1) \geqslant 1$. If $N=5,6,7$ or $8$, we have that $\ell_1=5$, and one can check directly that for these values, $\dfrac{(N+\ell_1 -2)!}{\ell_1!(N+1)!}(\ell_1^2-3\ell_1-N+1) \geqslant 1$.

\par
\vspace{0.5em}
Now, when $\ell \geqslant \lfloor \ell_0 \rfloor +2$, we have $\ell \geqslant \ell_0+1$ and by writing 
\[
\dfrac{(N+\ell -2)!}{\ell!(N+1)!}(\ell^2-3\ell-N+1) = \dfrac{(N+\ell-2)}{\ell}\dots \dfrac{(N+3)}{5}\cdot\dfrac{(N+2)}{4}\cdot\dfrac{(\ell + \ell_0)}{6}\cdot(\ell-\ell_0),
\]
we can see that all the factors of the right hand side are greater than $1$, as $\ell + \ell_0\geqslant 2\ell_0 \geqslant 8$ when $N\geqslant 5$. Hence, the claim is proved. 
\par
\vspace{0.5em}
Thus, for $\ell > \dfrac{3+\sqrt{4N+5}}{2}$, if $\displaystyle {N+\ell \choose N} \leqslant (N+1)s$, since 
    $
    \dfrac{{N+\ell \choose N}}{N+1} \geqslant \dfrac{{N-1+\ell \choose N-1}}{N} + \dfrac{{N-1+\ell-1 \choose N-1}}{N} +1,
    $
    we have that
    \[
    s\geqslant \left\lceil \dfrac{{N+\ell \choose N}}{N+1} \right\rceil \geqslant \left\lceil \dfrac{{N-1+\ell \choose N-1}}{N} + \dfrac{{N-1+\ell-1 \choose N-1}}{N} \right\rceil +1 \geqslant \left\lceil \dfrac{{N-1+\ell \choose N-1}}{N} \right\rceil + \left\lceil \dfrac{{N-1+\ell-1 \choose N-1}}{N} \right\rceil.
    \]
    Therefore, we can split $s$ into $r_1 \geqslant \left\lceil \dfrac{{N-1+\ell \choose N-1}}{N} \right\rceil$ and $r_2 \geqslant \left\lceil \dfrac{{N-1+\ell-1 \choose N-1}}{N} \right\rceil$. Considering $r_1$ and $r_2$ very general points in $\PP^{N-1}$, by the induction hypothesis, we have 
    \[
    \ahat(\PP^{N-1},r_1)\geqslant \dfrac{N-1+\ell}{N}=a_1 \text{  and  } \ahat(\PP^{N-1},r_2)\geqslant \dfrac{N-2+\ell}{N}=a_2.
    \]
    Note that $a_1>1$ and $a_2\leqslant2$ as $\ell \leqslant N-1$, hence, by Lemma \ref{lemma: Waldschmidtdecomp}, we get
    \[
    \ahat(\PP^{N},s) \geqslant \left( 1- \dfrac{N}{N-1+\ell}\right )\dfrac{N-2+\ell}{N} +1 \geqslant \dfrac{N+\ell}{N+1},
    \]
    where the last inequality is equivalent to $\ell^2-3\ell +2 \geqslant 0$, which is true for all $\ell$. This gives the desired inequality for generic points as well by Lemma \ref{lemma: known inequalities of Waldschmidt constant}$(iv)$.
\end{proof}



\section{ Ideal Containment and Demailly's Conjecture for Sufficiently Many General Points}
\label{sec.>2^NGen}
In this section, we will show that Demailly's Conjecture when $m=2$ holds for any set of at least $2^N$ many general points when $N\geqslant 6$. Our method is via ideal containment.
Harbourne-Huneke proposed the following containment for the defining ideal $I$ of a set of points in $\PP^N$. 
\[I^{(r(m+N-1))} \subset \mm^{r(N-1)}\left(  I^{(m)} \right), \text{ for all } r \text{ and } m \geqslant 1. \] 
The aforementioned containment implies Demailly's Conjectural bound. Our aim is to prove that for the defining ideal $I$ of $s$ generic points, there exists an integer $r$ such that we have the following stronger containment
\[
I^{(r(N+1)-N+1)} \subseteq \mm^{r(N-1)}\left ( I^{(2)} \right )^r,
\]
then follow the strategy in \cite[Theorem 2.9]{bghn2022demailly}. To prove this containment, using the same idea as in \cite[Lemma 2.6]{BCH2014}, it suffices to show that  $\alpha(I^{(r(N+1)-N+1)}) \geqslant r\reg(I^{(2)}) +r(N-1)$.  

\begin{theorem}
    \label{thm.neededineq}
     Let $I \subset R=\kk[\PP^N]$ be the defining ideal of $s$ generic points. If 
     \[
     \widehat{\alpha}(I) \geqslant \frac{\reg(I^{(2)})+N-1}{N+1},
     \]
     then there exists an integer $r$ such that 
\[
I^{(r(N+1)-N+1)} \subseteq \mm^{r(N-1)}\left ( I^{(2)} \right )^r.
\]
\end{theorem}

\begin{proof}
    We have 
    \[
    \lim_{r\rightarrow \infty} \frac{\alpha(I^{(r(N+1)-N+1)})}{r(N+1)} = \lim_{r\rightarrow \infty} \frac{\alpha(I^{(r(N+1)-N+1)})}{r(N+1)-N+1} = \widehat{\alpha}(I) \geqslant \frac{\reg(I^{(2)})+N-1}{N+1}.
    \]
    Thus, there exists $r$ such that, 
    \[
    \frac{\alpha(I^{(r(N+1)-N+1)})}{r(N+1)} \geqslant  \frac{\reg(I^{(2)})+N-1}{N+1},
    \]
    hence, $\alpha(I^{(r(N+1)-N+1)}) \geqslant r\reg(I^{(2)})+r(N-1).$ This implies $I^{(r(N+1)-N+1)} \subseteq \mm^{r(N-1)}\left ( I^{(2)} \right )^r$.
\end{proof}

\begin{lemma}
    \label{lem.Demaillygeneral}
    For a fix $\ell \in \NN$, let $I$ be the defining ideal of a set of $s$ generic points or very general points such that $\displaystyle {N+\ell \choose N} < (N+1)s \leqslant {{N+\ell+1} \choose N}$. If $\displaystyle \ahat(I)\geqslant \dfrac{N+\ell+1}{N+1},$ then the inequality,
    $$\ahat(I)\geqslant \dfrac{\reg(I^{(2)})+N-1}{N+1},$$
    holds. Moreover, Demailly's Conjecture when $m=2$ holds for a set of $s$ general points in $\PP^N$.
\end{lemma}

\begin{proof}
    If ${N+\ell \choose N} < (N+1)s \leqslant {{N+\ell+1} \choose N}$ then by Lemma \ref{lem.regbound},  $\reg \left( I^{(2)} \right) \leqslant \ell+2$. Thus the inequality $\ahat(I)\geqslant \frac{N+\ell+1}{N+1}$ implies the inequality $\ahat(I)\geqslant \frac{\reg(I^{(2)})+N-1}{N+1}$ for all $\ell$. By Theorem \ref{thm.neededineq}, the stable containment $I^{(r(N+1)-N+1)} \subseteq \mm^{r(N-1)}\left ( I^{(2)} \right )^r$ holds for $s$ generic points, hence, by the proof of \cite[Theorem 2.9]{bghn2022demailly}, the Demailly's Conjecture when $m=2$ holds for a set of $s$ general points in $\PP^N$.
\end{proof}

\begin{lemma}\label{lemma: how far to go}
The following inequalities hold:
\begin{enumerate}[label=(\alph*)]
    \item If $N\geqslant 5$, then $4^N(N+1) \leqslant \displaystyle {3N+2 \choose N}$.
    \item If $N\geqslant 11$, then $3^N(N+1) \leqslant \displaystyle {2N+2 \choose N}$.
    \item If $N\geqslant 7$, then $3^N(N+1) \leqslant \displaystyle {2N+3 \choose N}$. 
\end{enumerate}
\end{lemma}

\begin{proof}
$(a).$ For $N=5$, we can check that $4^5\cdot6=6144$, and ${3\cdot5+2 \choose 5}=6188$. Now assume that $4^N(N+1) \leqslant {3N+2 \choose N}$ holds for an arbitrary $N\geqslant 5$.  
Now \begin{align*}
3(N+1)+2 \choose {N+1} &=\dfrac{(3N+5)(3N+4)(3N+3)}{(N+1)(2N+4)(2N+3)} {3N+2 \choose N}\\
                                        &\geqslant \dfrac{(3N+5)(3N+4)(3N+3)}{(N+1)(2N+4)(2N+3)} \cdot 4^N(N+1)\\
                                        &=\dfrac{(3N+5)(3N+4)(3N+3)}{(2N+4)(2N+3)}4^N\\
                                        &\geqslant 4^{N+1}(N+2). 
\end{align*} 
Thus the result follows by induction.  
The proofs for parts $(b)$ and $(c)$ are essentially the same.
\end{proof}

\begin{theorem}
\label{thm.Nge62^Nto4^N}
 Let $N\geqslant 6$, and $I$ be the defining ideal of $s$ generic points in $\PP^N$, where $2^N \leqslant s \leqslant 4^N$. Then the following holds  
$$\ahat(I) \geqslant \dfrac{\reg(I^{(2)})+N-1}{N+1} .$$ 
\end{theorem}
\begin{proof}
If the number of points satisfies $3^N \leqslant s \leqslant4^N$, then by Lemma \ref{lemma: how far to go} $(a)$ we have $\displaystyle s(N+1) \leqslant {N+2N+2 \choose N}$. Hence, by Lemma \ref{lem.regbound}, we have $\reg(I^{(2)}) \leqslant 2N+3$. The inequality follows as in this case, by Lemma \ref{lemma: known inequalities of Waldschmidt constant}, $\ahat(I)\geqslant 3 > \dfrac{N-1+2N+3}{N+1}$.\par
\vspace{0.5em}
Now we consider the case where $2^N \leqslant s \leqslant 3^N$. By Lemma \ref{lem.Demaillygeneral}, we only need to prove that
 \[\ahat(I)\geqslant \dfrac{N-1+\ell+2}{N+1}=\dfrac{N+\ell+1}{N+1} \text{ whenever } {N+\ell \choose N} < (N+1)s \leqslant {{N+\ell+1} \choose N}. \quad \tag{3} \label{ineq.lemma2^Nto3^N} \] 
	Note that $3^N(N+1) \leqslant {2N+2 \choose N}$,  for  all $N\geqslant 11$ by Lemma \ref{lemma: how far to go} $(b)$.  Thus we have $s(N+1) \leqslant {N+N+2 \choose N}$ whenever $N\geqslant 11$, and  by Lemma \ref{lem.regbound}, we get $\reg(I^{(2)}) \leqslant N+3$. Hence, by Lemma \ref{lemma: known inequalities of Waldschmidt constant}, we get  $$\ahat(I)\geqslant 2 = \frac{N-1+N+3}{N+1}, \text{ for all } N \geqslant 11. $$
	
\noindent	Now we consider the cases when $6 \leqslant N \leqslant 10$.\\ 

	\noindent\textsf{Case 1:} When $N=6$, it is enough to check inequality (\ref{ineq.lemma2^Nto3^N}) for $4 \leqslant\ell \leqslant 9$ (and $2^6 \leqslant s \leqslant3^6$). 
	\begin{enumerate}[label=(\roman*)]
		\item When $\ell=6$ or $7$, since $s\geqslant 2^6$, then by Lemma \ref{lemma: known inequalities of Waldschmidt constant}, we have $\ahat(I(1^{\times s})) \geqslant 2 \geqslant \dfrac{6+\ell+1}{6+1}$. 
		\item When $\ell=8$, then $s\geqslant 429=2^6\cdot 6+45$. We first show that $\ahat\left( I\left( 2^{\times 6}, 1 ^{\times 45}\right)\right)\geqslant \dfrac{22}{10}$. Assume that $I \left(20m^{\times 6}, 10m^{\times 45} \right)_{22m-1} \neq 0$. Let $k=5(22m-1)-(6\cdot 20m+10m)=-20m-5$ and applying Lemma \ref{lemma: reduction of multiplicity of points}, we get $I\left(10m^{\times 44}\right)_{2m-6} \neq 0$, which is a contradiction. Thus $\ahat\left( I\left( 2^{\times 6}, 1 ^{\times 45}\right)\right)\geqslant \dfrac{22}{10}$. Hence, by  \cite[Proposition 3.6]{SankhoThaiChudnovsky} we have $\ahat\left(I(1^{\times s})\right) \geqslant \ahat\left( I\left( 1^{\times 6 \cdot 2^6}, 1 ^{\times 45}\right)\right) \geqslant \ahat\left( I\left( 2^{\times 6}, 1 ^{\times 45}\right)\right)\geqslant \dfrac{22}{10}>\dfrac{6+8+1}{6+1}.$
		\item When $\ell=9$, then $s\geqslant 715 \geqslant 7 \times 2^6$. Thus by \cite[Theorem 3.2]{SankhoThaiChudnovsky} $$\ahat \left( I(1^{\times s})\right)\geqslant \ahat \left( I(1^{\times 7 \cdot 2^6})\right)\geqslant 2\ahat(I(1^{\times 7}))=2 \cdot \dfrac{7}{6} >\dfrac{6+9+1}{6+1}. $$
	\end{enumerate}
	
	\noindent\textsf{Case 2:} When $7 \leqslant N\leqslant 10$, it is enough to check inequality (\ref{ineq.lemma2^Nto3^N}) for $2^N \leqslant s \leqslant3^N$ and $\ell \leqslant N+2$ as $3^N(N+1) \leqslant {2N+3 \choose N}$ for  all $N\geqslant 7$ by Lemma \ref{lemma: how far to go} (c).
		\begin{enumerate}[label=(\roman*)]
			\item When $\ell \leqslant N+1$, and $s\geqslant 2^N$ we have $\ahat(I(1^{\times s})) \geqslant 2 \geqslant \dfrac{N+\ell+1}{N+1}$. 
			\item When $\ell=N+2$, then we can check directly that the number of points satisfies $s \geqslant 2^N(N+1)$ with these values of $N$. Thus by \cite[Theorem 3.2]{SankhoThaiChudnovsky} $$\ahat \left( I(1^{\times s})\right)\geqslant \ahat \left( I(1^{\times 2^N(N+1)})\right)\geqslant 2\ahat\left(I\left(1^{\times (N+1)}\right)\right)=2 \cdot \dfrac{N+1}{N} >\dfrac{N+N+2+1}{N+1}. $$
		\end{enumerate}
	
\end{proof}

\begin{corollary}
    \label{cor.genericatleast2^N}
    Suppose that $N\geqslant 3$, and let $I$ be the defining ideal of $s$ generic points in $\PP^N$, where $s\geqslant 2^N$. Then   
    $$\ahat(I) \geqslant \dfrac{\reg(I^{(2)})+N-1}{N+1} .$$ 
\end{corollary}
\begin{proof}
    The case where $N=3,4$ and $5$ will be proved in Section \ref{sec.DemGen}. The case $N\geqslant 6$ follows directly from Theorem \ref{thm.Nge62^Nto4^N} and by the proof of \cite[Lemma 2.8]{bghn2022demailly} and \cite[Remark 2.10]{bghn2022demailly}.
\end{proof}

\begin{corollary}
    \label{cor.generalatleast2^N}
    Suppose that $N\geqslant 3$, and let $I$ be the defining ideal of $s$ general points in $\PP^N$, where $s\geqslant 2^N$. Then $I$ verifies Demailly's Conjecture when $m=2$, that is,  
    $$\ahat(I) \geqslant \dfrac{\alpha(I^{(2)})+N-1}{N+1} .$$ 
\end{corollary}
\begin{proof}
	The proof goes along the same line as that of \cite[Theorem 2.9]{bghn2022demailly}.
\end{proof}



\section{Ideal Containment and Demailly's Conjecture for General Points in Low Dimension}
\label{sec.DemGen}

In this section, we shall show Demailly's Conjecture when $m=2$ for any set of general points in $\PP^3, \PP^4$ and $\PP^5$, except four numbers of points in $\PP^5$ where the result is still not known. Note that the results in this section are used for the base cases in the proof of Theorem \ref{thm.Demaillyverygeneral}. 
Since Conjecture \ref{conjecture.Demailly2} is known to be true for small number of general points  \cite{NagelTrokInterpolation} and sufficiently many general points \cite{bghn2022demailly}, we will only consider the remaining cases as discussed below. 


\begin{theorem}\label{thm.needediqP3}
Let $I$ be the defining ideal of $s$ generic points in $\PP^3$. If  $6\leqslant s \leqslant 6^3=216$, then 
$$ \ahat(I) \geqslant  \frac{2+\reg(I^{(2)})}{4}.$$
\end{theorem}
\begin{proof}
 Since $ 20={{3+3} \choose 3}< 
4\cdot 6  < {3+4 \choose 3}=35$, and $816={3+15 \choose 3}< 4\cdot6^3  < {3+16\choose 3}=969$, then by Lemma \ref{lem.Demaillygeneral}, it suffices to show that for $3\leqslant\ell \leqslant15$,
$$\ahat(I)\geqslant \dfrac{N-1+\ell+2}{N+1}=\dfrac{4+\ell}{4},\quad \quad \quad  \quad \quad \quad$$ 
where $I$ is the defining ideal of $s$ generic points with ${3+\ell \choose 3}<4s \leqslant {3+\ell+1 \choose 3}$. 
We consider case by case as follows.
\par
\vspace{0.5em}

\noindent\textsf{Case 5.1.1: $\ell = 3, s=6,7, 8$.}
By \cite[Proposition B.1.1]{brianlinear}, we have
$\ahat(I)\geqslant \ahat(6) \geqslant \dfrac{5}{3}>\dfrac{4+3}{5}$.\par
\vspace{0.5em}

\noindent\textsf{Case 5.1.2: $\ell = 4, 9\leqslant s\leqslant 14$.}
Since $s>2^3$, we have $\ahat(I) \geqslant 2$, thus, $\ahat(I) \geqslant 2 = \dfrac{4+4}{4}$. 
\par
\vspace{0.5em}

\noindent\textsf{Case 5.1.3: $\ell = 5, 15\leqslant s\leqslant 21$.}
By \cite[Proposition 11]{Dumnicki2015}, $\ahat(I) \geqslant \ahat(14)\geqslant  \dfrac{7}{3}> \dfrac{4+5}{4}$.
\par
\vspace{0.5em}

\noindent\textsf{Case 5.1.4: $\ell = 6, 22\leqslant s\leqslant 30$.}
By \cite[Proposition 11]{Dumnicki2015}, $\ahat(I) \geqslant \ahat(21) \geqslant  \dfrac{8}{3}>\dfrac{4+6}{4}$.
\par
\vspace{0.5em}

\noindent\textsf{Case 5.1.5: $\ell = 7,8, 30\leqslant s\leqslant55$.}
Since $s > 3^3$, we have $\ahat(I) \geqslant 3$, thus $\ahat(I) \geqslant 3 = \dfrac{4+8}{4}$.
\par
\vspace{0.5em}

\noindent\textsf{Case 5.1.6: $\ell = 9, 56\leqslant s\leqslant 71$.}
The result holds as
\[
\ahat(I) \geqslant \ahat(56) \geqslant 2\ahat(7) = \frac{56}{15} > \frac{4+9}{4}.
\]
where $\ahat(7) \geqslant \dfrac{28}{15}$ holds by \cite[Proposition 11]{Dumnicki2015} and $\ahat(56) \geqslant 2\ahat(7)$ holds by \cite[Theorem 3.2]{SankhoThaiChudnovsky}.
\par
\vspace{0.5em}

\noindent\textsf{Case 5.1.7: $10 \leqslant\ell \leqslant 12, 72\leqslant s\leqslant 140$.}
Since $s> 4^3$, we have $\ahat(I) \geqslant 4$, thus, $\ahat(I) \geqslant 4 = \dfrac{4+12}{4}$.
\par
\vspace{0.5em}

\noindent\textsf{Case 5.1.8: $13 \leqslant\ell \leqslant15, 141\leqslant s\leqslant216$.}
Since $s> 5^3$, we have $\ahat(I) \geqslant 5$, thus, $\ahat(I) \geqslant 5 > \dfrac{4+15}{4}$.

\end{proof}



\begin{theorem}\label{thm.needediqP4}
	Let $I$ be the defining ideal of $s$ generic points in $\PP^4$. If  $8\leqslant s \leqslant5^4=625$, then 
	$$ \ahat(I) \geqslant  \frac{3+\reg(I^{(2)})}{5}.$$
\end{theorem}
\begin{proof}
	Since $35={4+3\choose 4} < 5\cdot 8 < 70={ 4+4 \choose 4}$ and $ 3060={4+14 \choose 4}< 5 \cdot625 < {4+15 \choose 4}=3876$, then by Lemma \ref{lem.Demaillygeneral} it suffices to show that for $3\leqslant\ell \leqslant14$,
	$$\ahat(I)\geqslant \dfrac{N-1+\ell+2}{N+1}=\dfrac{5+\ell}{5},\quad \quad \quad  \quad \quad \quad$$ 
	where $I$ is the defining ideal of $s$ generic points with ${4+\ell \choose 4}<5s \leqslant {4+\ell+1 \choose 4}$.

We again consider case by case as follows.
\vspace{0.5em}

\noindent\textsf{Case 5.2.1: $\ell = 3, 8\leqslant s\leqslant14$.}
By \cite[Lemma 3.9]{SankhoThaiChudnovsky},
$\ahat(I) \geqslant \ahat(8) \geqslant \dfrac{8}{5}=\dfrac{5+3}{5}$.
\par
\vspace{0.5em}

\noindent\textsf{Case 5.2.2: $\ell = 4, 15\leqslant s \leqslant 25$.}
For finding the right lower bound for $\ahat(I)$, we will use the similar trick to what was used in \cite[Lemma 4.5]{SankhoThaiChudnovsky}. Since $\ahat({1^{\times 6}})\geqslant \dfrac{6}{4}=\dfrac{30}{20}$, we have 
\[
\displaystyle I((20m)^{\times 6})_{30m-1} = 0, \text{  for all  } m.
\]
We claim that 
\[
\displaystyle I\left((20m)^{\times 9},(30m)^{\times 1}\right)_{36m-1} = 0, \forall m. 
\]
Applying  Lemma \ref{lemma: reduction of multiplicity of points} for three times with largest possible multiplicities proves the claim. Indeed, each time the reductor $k=3d-\sum_{i=1}^{5}m_i$ affects the coefficients of $m$ in $m_i$ and $d$ significantly. In Table \ref{table:1} we capture the  coefficient of $m$ in $k$, and change of coefficients of $m$ in  $m_i$ and $d$.  The change in the constant term does not affect so we omit them in the table. 
\begin{table}[h!]
	\centering
	\begin{tabular}{||c |c c c c c c  c c c c| c ||} 
		\hline
		d & $m_1$ & $m_2$ & $m_3$ & $m_4$ & $m_5$ & $m_6$ & $m_7$ & $m_8$& $m_9$ & $m_{10}$& $k$\\ [0.5ex] 
		\hline\hline
		$36$ & $\underline{30}$ & $\underline{20}$ & $\underline{20}$ &$\underline{20}$  & $\underline{20}$ & $20$& $20$ & $20$ & $20$  &  $20$ & $-2$ \\ 
		$34$ & $\underline{28}$ &$18$ & $18$ & $18$  &  $18$ & $\underline{20}$ & $\underline{20}$ &$\underline{20}$  & $\underline{20}$ & $20$&  $-6$ \\ 
		$28 $ & $\underline{22}$ & $ \underline{18}$& $\underline{18}$& $\underline{18}$& $18$& ${14}$ & $ {14}$& $ {14}$& $14$& $\underline{20}$& $-12$\\
		$16 $ & $\underline{10}$ & $ \underline{6}$& $\underline{6}$& $\underline{6}$& $18$& ${14}$ & $ {14}$& $ {14}$& $14$& $\underline{8}$& \\
		\hline
	\end{tabular}
	\caption{Reduction table for $I\left((20m)^{\times 9},(30m)^{\times 1}\right)_{36m-1}$}
	\label{table:1}
\end{table}\\
From the last row of Table \ref{table:1}, since, the coefficient of  $m$ in the degree $d$ is less than that of  the fifth point,  we conclude  that  $\displaystyle I\left((20m)^{\times 9},(30m)^{\times 1}\right)_{36m-1} = 0, \forall m$.
Now, by applying the Lemma \ref{lemma: adding multiplicities Dumnicki} one time, we get $\displaystyle I((20m)^{\times 15})_{36m-1} = 0$, for all $m$,
therefore, $\ahat(I)\geqslant \ahat(15) \geqslant \dfrac{36}{20} = \dfrac{5+4}{5}$.
\par
\vspace{0.5em}

\noindent\textsf{Case 5.2.3: $\ell = 5, 26\leqslant s\leqslant 42$.}
Since $s \geqslant 2^4$, we have $\ahat(I) \geqslant 2$, thus $\ahat(I) \geqslant 2=\dfrac{5+5}{5}$.
\par
\vspace{0.5em}

\noindent\textsf{Case 5.2.4: $\ell = 6, 43\leqslant s\leqslant66$.} 
Since $\ahat({1^{\times 6}})\geqslant \dfrac{6}{4}=\dfrac{30}{20}$, we have 
\[
\displaystyle I((20m)^{\times 6})_{30m-1} = 0, \text{  for all  } m.
\]
We claim that $\displaystyle I\left((20m)^{\times 1},(30m)^{\times 7}\right)_{44m-1} = 0, \forall m.$ In fact, if it is not the case for some $m$, by applying Lemma \ref{lemma: reduction of multiplicity of points} with $k=3(44m-1)-5\cdot 30m=-18m-3$, we can reduce to
\[
\displaystyle I\left((20m)^{\times 1},(30m)^{\times 2},(12m-3)^{\times 5}\right)_{26m-4} \ne 0,
\]
which is a contradiction as $26m-4<30m$. Now, by applying the Lemma \ref{lemma: adding multiplicities Dumnicki} seven times, we get $\displaystyle I((20m)^{\times 43})_{44m-1} = 0$, for all $m$, therefore, $\ahat(I) \geqslant \ahat(43) \geqslant \dfrac{44}{20} = \dfrac{5+6}{5} $.
\par
\vspace{0.5em}

\noindent\textsf{Case 5.2.5: $\ell = 7, 67\leqslant s\leqslant 99$.} 
Again, since $\ahat({1^{\times 6}})\geqslant \dfrac{6}{4}=\dfrac{30}{20}$, we have 
\[
\displaystyle I((20m)^{\times 6})_{30m-1} = 0, \text{  for all  } m.
\]
We claim that $\displaystyle I\left((20m)^{\times 1},(30m)^{\times 11}\right)_{48m-1} = 0, \forall m.$ Similarly, we apply Lemma \ref{lemma: reduction of multiplicity of points} for two times with largest possible multiplicities proves the claim. In Table \ref{table:2} we capture the  coefficient of $m$ in $k$, and change of coefficients of $m$ in  $m_i$ and $d$. Again, the change in the constant term does not affect so we omit those in the table. 
\begin{table}[h!]
	\centering
	\begin{tabular}{||c |c c c c c c  c c c c c c| c ||} 
		\hline
		d & $m_1$ & $m_2$ & $m_3$ & $m_4$ & $m_5$ & $m_6$ & $m_7$ & $m_8$& $m_9$ & $m_{10}$&$m_{11}$&$m_{12}$& $k$\\ [0.5ex] 
		\hline\hline
		$48$ & $\underline{30}$ & $\underline{30}$ & $\underline{30}$ &$\underline{30}$  & $\underline{30}$ & $30$& $30$ & $30$ & $30$  &  $30$ & $30$ & $20$ & $-6$ \\ 
		$42$ & $24$ &$24$ & $24$ & $24$  & $24$&  $\underline{30}$ & $\underline{30}$ & $\underline{30}$ &$\underline{30}$  & $\underline{30}$ &$30$& $20$& $-24$ \\ 
		$18 $ & $24$ &$24$ & $24$ & $24$  & $24$&  $6$ & $ 6$& $6$& $6$& $6$& $30$& $20$ & $-12$\\
		\hline
	\end{tabular}
	\caption{Reduction table for $I\left((20m)^{\times 1},(30m)^{\times 11}\right)_{48m-1}$}
	\label{table:2}
\end{table}\\
From the last row of Table \ref{table:2}, since, the coefficient of  $m$ in the degree $d$ is less than that of  the first point,  we conclude  that  $\displaystyle I\left((20m)^{\times 1},(30m)^{\times 11}\right)_{48m-1} = 0, \forall m.$ Now, by applying the Lemma \ref{lemma: adding multiplicities Dumnicki} eleven times, we get $\displaystyle I((20m)^{\times 67})_{48m-1} = 0$, for all $m$,
therefore,
$\ahat(I) \geqslant \ahat(67) \geqslant \dfrac{48}{20} = \dfrac{5+7}{5}$. 
\par

\vspace{0.5em}

\noindent\textsf{Case 5.2.6: $8\leqslant\ell \leqslant 10, 100\leqslant s\leqslant 273$.}
Since $s > 3^4$, we have $\ahat(I) \geqslant 3$, thus $\ahat(I) \geqslant 3 = \dfrac{5+10}{5}$.
\par
\vspace{0.5em}

\noindent\textsf{Case 5.2.7: $11\leqslant\ell \leqslant 14, 274 \leqslant s\leqslant 625$.}
Since $s > 4^4$, we have $\ahat(I) \geqslant 4$, thus $\ahat(I) \geqslant 4 > \dfrac{5+14}{5}$.

\end{proof}

\begin{theorem}\label{thm.needediqP5}
	Let $I$ be the defining ideal of $s$ generic points in $\PP^5$. If  $8\leqslant s \leqslant 4^5=1024$, except $10 \leqslant s \leqslant 13$ then 
	$$ \ahat(I) \geqslant  \frac{4+\reg(I^{(2)})}{6}.$$
 The result for $10 \leqslant s \leqslant 13$ is not known.
\end{theorem}
\begin{proof}
	Since $21={5+2\choose 5} < 6\cdot 8 < 56={ 5+3 \choose 5}$ and $ 4368={5+11 \choose 5}< 6 \cdot 1024 < {5+12 \choose 5}=6188$, then by Lemma \ref{lem.Demaillygeneral} it suffices to show that for $2\leqslant\ell \leqslant 11$,
$$\ahat(I)\geqslant \dfrac{N-1+\ell+2}{N+1}=\dfrac{6+\ell}{6},\quad \quad \quad  \quad \quad \quad$$ 
where $I$ is the defining ideal of $s$ generic points with ${5+\ell \choose 5}<6s \leqslant {5+\ell+1 \choose 5}$.
We consider case by case as follows.	

\par
\vspace{0.5em}

\noindent\textsf{Case 5.3.1: $\ell=2, s=8,9$.} 
By \cite[Proposition B.1.1]{brianlinear}, we have 
$\ahat(\PP^5, 8) \geqslant \dfrac{5+2}{5} >  \dfrac{6+2}{6}.$

\noindent\textsf{Case 5.3.2: $\ell=3$ and $14 \leqslant s\leqslant 21$.}
Applying Lemma \ref{lemma: Waldschmidtdecomp} using Theorem \ref{thm.needediqP4} as follows: 

pick $\displaystyle r_1=8, \text{ since } {4+3\choose 4 } < 5r_1 \leqslant {4+4 \choose 4}, \text{ we can pick } a_1 = \frac{5+3}{5}=\frac{8}{5}, \text{ and }$

$\displaystyle r_2=6, \text{ using the fact that } \ahat(\PP^N,N+2)\geqslant \frac{N+2}{N}, \text{ we can pick } a_2 = \frac{4+2}{4}=\frac{6}{4},$
thus, we have
$$\ahat(\PP^5, 14) \geqslant (1-\dfrac{1}{a_1})a_2+1=(1-\dfrac{5}{8})\dfrac{6}{4}+1=1.5625 > 1.5= \dfrac{6+3}{6}.$$

\noindent\textsf{Case 5.3.3: $\ell=4$ and $22\leqslant s \leqslant 42$.}
Similarly, applying Lemma \ref{lemma: Waldschmidtdecomp} using Theorem \ref{thm.needediqP4}  as follows:

pick $\displaystyle r_1=15, \text{ since } 70={4+4\choose 4 } < 5r_1 \leqslant {4+5 \choose 4}, \text{ we can pick }  a_1 = \frac{5+4}{5}=\frac{9}{5}, \text{ and}$

$\displaystyle  r_2=7 , \text{ using the fact that } \ahat(\PP^N,N+2)\geqslant \frac{N+2}{N}, \text{ we can pick }  a_2 = \frac{4+2}{4}=\frac{6}{4},$
thus, we have
$$\ahat(\PP^5, 22) \geqslant (1-\dfrac{1}{a_1})a_2+1=(1-\dfrac{5}{9})\dfrac{6}{4}+1= \dfrac{6+4}{6}.$$

\noindent\textsf{Case 5.3.4: $\ell=5,6$ and ${43} \leqslant s \leqslant 132$.}
Since $ s \geqslant 42>2^5$ we have, $\ahat(\PP^5,s)\geqslant 2 > \dfrac{6+6}{6}$.

\noindent\textsf{Case 5.3.5: $\ell=7$ and ${133} \leqslant s \leqslant 214$.} 
Again, applying Lemma \ref{lemma: Waldschmidtdecomp} using Theorem \ref{thm.needediqP4}  as follows:

pick $\displaystyle  r_1=67, \text{ since } 330={4+7\choose 4} < 5r_1 \leqslant {4+8 \choose 4}, \text{ we can pick } a_1 = \dfrac{5+7}{5}=\dfrac{12}{5},$

$\displaystyle  r_2=43, \text{ since }  210={4+6\choose 4 } {<} 5r_2 \leqslant {4+5 \choose 4}, \text{ we can pick }  a_2 = \dfrac{5+6}{5}=\dfrac{11}{5}, \text{ and}$

$\displaystyle  r_3=15, \text{ since } 70={4+4\choose 4 } {<} 5r_3 \leqslant {4+5 \choose 4}, \text{ we can pick }  a_3 = \dfrac{5+4}{5}=\dfrac{9}{5},$
thus, we have 
$$\ahat(\PP^5, 133) \geqslant  \ahat(\PP^5, 125)\geqslant  \left(1- \left(\dfrac{1}{a_1}+\dfrac{1}{a_2} \right) \right)a_3+1=\left(1-\left(\dfrac{5}{11} +\frac{5}{12}\right)\right)\dfrac{9}{5}+2 >2.23>\dfrac{6+7}{6}.$$

\noindent\textsf{Case 5.3.6: $\ell=8$ and $215 \leqslant s \leqslant 333 $.} 
By \cite[Theorem 3.2]{SankhoThaiChudnovsky}, we have 
\[
\ahat(\PP^5,215)\geqslant \ahat(\PP^5,6\cdot 32) \geqslant 2\ahat(6) \geqslant 2\cdot \frac{6}{5} > \frac{6+8}{6}.
\]

\noindent\textsf{Case 5.3.8: $9\leqslant \ell \leqslant 11$ and ${334} \leqslant s \leqslant 1024$.} 
Since $s \geqslant 334 >3^5$, we have $\ahat(\PP^5,s)\geqslant 3 > \dfrac{6+11}{6}$.

\end{proof}
As a consequence of  Theorems \ref{thm.neededineq}, \ref{thm.needediqP3}, \ref{thm.needediqP4}, and \ref{thm.needediqP5}, we deduce the following corollary regarding general points in $\PP^N$ for $N=3,4$, and $5$. 
\begin{corollary}\label{cor.Demailly2.general.P345}
	Let $I$ be the defining ideal of $s$ general points in $\PP^N$.  If  $s$ satisfies the following conditions 
	\begin{enumerate}[label=(\roman*)]
\item  $6\leqslant s \leqslant 6^3=216$, when $N=3$;
\item  $8\leqslant s \leqslant 5^4=624$, when $N=4$;
\item  $s=8,9$ and $14 \leqslant s \leqslant 4^5=1024$, when $N=5$;
	\end{enumerate}  then 
	$$ \ahat(I) \geqslant  \frac{\alpha(I^{(2)})+N-1}{N+1}.$$

\end{corollary}
\begin{proof} 
	The proof has the same structure as \cite[Theorem 2.9]{bghn2022demailly}.
\end{proof}
\begin{remark}
    \label{rem.DemaillyConclusion} 
    {Combined results in \cite{bghn2022demailly,localeffectivity,NagelTrokInterpolation}, Demailly's Conjecture when $m=2$ now holds for any set of very general points, any set of at least $2^N$ general points, and any set of points in $\PP^3,\PP^4$, and $\PP^5$, except for $10\leqslant s \leqslant 13$  in $\PP^5$.}  The case $10 \leqslant s \leqslant 13$ in $\PP^5$ is not known.
 
\end{remark}



\end{document}